\documentclass[11pt]{amsart}

\usepackage{hyperref}

\usepackage{amssymb,latexsym, amsthm, enumerate, mathptmx}

\usepackage{tikz}
\usetikzlibrary{arrows}

\newtheorem{theorem}{Theorem}[section]
\newtheorem{lemma}[theorem]{Lemma}
\newtheorem{proposition}[theorem]{Proposition}
\newtheorem{corollary}[theorem]{Corollary}
\newtheorem{claim}{Claim}
\newtheorem{conjecture}[theorem]{Conjecture}

\newtheoremstyle{definition}% name
  {6pt}%      Space above
  {6pt}%      Space below
  {}%         Body font
  {}%         Indent amount (empty = no indent, \parindent = para indent)
  {\bfseries}% Thm head font
  {.}%        Punctuation after thm head
  {.5em}%     Space after thm head: " " = normal interword space;
  {}%

\theoremstyle{definition}
\newtheorem{definition}[theorem]{Definition}

\newtheoremstyle{remark}% name
  {6pt}%      Space above
  {6pt}%      Space below
  {}%         Body font
  {}%         Indent amount (empty = no indent, \parindent = para indent)
  {\bfseries}% Thm head font
  {.}%        Punctuation after thm head
  {.5em}%     Space after thm head: " " = normal interword space;
  {}%

\theoremstyle{remark}
\newtheorem{remark}[theorem]{Remark}

\newtheoremstyle{example}% name
  {6pt}%      Space above
  {6pt}%      Space below
  {}%         Body font
  {}%         Indent amount (empty = no indent, \parindent = para indent)
  {\bfseries}% Thm head font
  {.}%        Punctuation after thm head
  {.5em}%     Space after thm head: " " = normal interword space;
  {}%

\theoremstyle{example}
\newtheorem{example}[theorem]{Example}

\makeatletter
\renewcommand\@makefntext[1]{%
\setlength\parindent{1em}% 
\noindent
\makebox[1.8em][r]{}{#1}}
\makeatother

\DeclareMathOperator{\codim}{codim}
\DeclareMathOperator{\pnt}{\raise 0.5mm \hbox{\large\bf.}}
\DeclareMathOperator{\lpnt}{\hbox{\large\bf.}}

\title{Broken circuit complexes and hyperplane arrangements}

\author{Le Van Dinh}
\address{Universit\"at Osnabr\"uck, Institut f\"ur Mathematik, 49069 Osnabr\"uck, Germany}
\email{dlevan@uos.de}

\author{Tim R\"omer}
\address{Universit\"at Osnabr\"uck, Institut f\"ur Mathematik, 49069 Osnabr\"uck, Germany}
\email{troemer@uos.de}

\begin{document}

\begin{abstract}
We study Stanley-Reisner ideals of broken circuits complexes and characterize those ones admitting a linear resolution or being complete intersections. These results will then be used to characterize arrangements whose Orlik-Terao ideal has the same properties. As an application, we improve a result of Wilf on upper bounds for the coefficients of the chromatic polynomial of a maximal planar graph. We also show that for an ordered matroid with disjoint minimal broken circuits, the supersolvability of the matroid is equivalent to the Koszulness of its Orlik-Solomon algebra.
\end{abstract}

\maketitle

\section{Introduction}

Let $V$ be a vector space of dimension $r$ over a field ${K}$. Denote by $V^*$ the dual space of $V$. Let $\mathcal{A}=\{H_1,\ldots,H_n\}$ be an essential central hyperplane arrangement in $V$. Then the underlying matroid $M(\mathcal{A})$ of $\mathcal{A}$ has rank $r$ and there are linear forms $\alpha_i\in V^*$ such that $\ker \alpha_i=H_i$ for $i=1,\ldots,n.$ Let $X=V\backslash\bigcup_{i=1}^n H_i$ be the complement of the hyperplanes. It is well-known that when ${K}=\mathbb{C}$ the cohomology ring $H^{\pnt}(X,\mathbb{Z})$ of $X$ depends only on the matroid $M(\mathcal{A})$: $H^{\pnt}(X,\mathbb{Z})$ is isomorphic as a graded $\mathbb{C}$-algebra to the Orlik-Solomon algebra $\mathbf{A}(\mathcal{A})$ of $\mathcal{A}$; see \cite[Theorem 5.2]{OS}. Here the Orlik-Solomon algebra $\mathbf{A}(\mathcal{A})$ is defined as the quotient of the standard graded exterior algebra $E=\mathbb{Z}\langle e_1,\ldots,e_n\rangle$ by the Orlik-Solomon ideal $J(\mathcal{A})$ of $\mathcal{A}$ which
  is generated
by all elements of the form
$$\partial e_{i_1\ldots i_m}=\sum_{t=1}^m(-1)^{t-1}e_{i_1}\cdots \widehat{e_{i_t}}\cdots e_{i_m},$$
where $\{H_{i_1},\ldots,H_{i_m}\}$ is a dependent subset of $\mathcal{A}$, i.e., $\alpha_{i_1},\ldots,\alpha_{i_m}$ are linearly dependent. Since their appearance in \cite{OS}, the Orlik-Solomon algebra has been proved to be a very important algebraic object associated to an arrangement and it has been studied intensively; see, e.g., \cite{B,EPY,KR,OT,SS1,Y2} for details.

\footnotetext{
\begin{itemize}
\item[ ]
{\it Mathematics  Subject  Classification} (2010): 05E40, 13C40, 13D02,  13F55,  52B40, 52C35.
\item[ ]
{\it Key words and phrases}: broken circuit complex, complete intersection, hyperplane arrangement, matroid, Orlik-Terao algebra, resolution.
\end{itemize}
}

The Orlik-Terao algebra of $\mathcal{A}$, which was first introduced in \cite{OT2}, is a commutative analog of the  Orlik-Solomon algebra. For our purposes we follow the exposition of Schenck-Tohaneanu \cite{ST}. Let $S={K}[x_1,\ldots,x_n]$ be the polynomial ring in $n$ variables over ${K}$ ($n$ is the number of hyperplanes of $\mathcal{A}$). Define a ${K}$-linear map
$$\varphi:S_1=\bigoplus_{i=1}^n{K}x_i\rightarrow V^*,\ x_i\mapsto \alpha_i\ \text{ for }\ i=1,\ldots,n.$$
We call the kernel of this map the {\it relation space} and denote it by $F(\mathcal{A})$. Elements of $F(\mathcal{A})$ are called {\it relations}.  Observe that relations come from dependencies among hyperplanes in $\mathcal{A}$: If $\{H_{i_1},\ldots,H_{i_m}\}$ is a dependent subset of $\mathcal{A}$ and $a_t\in{K}$ are coefficients such that $\sum_{t=1}^ma_t\alpha_{i_t}=0$, then $r=\sum_{t=1}^ma_tx_{i_t}$ is a relation.

\begin{definition}
For each relation $r\in F(\mathcal{A})$, we write $r=\sum_{t=1}^ma_tx_{i_t}$ with $a_t\ne0$ for all $t=1,\ldots,m$. Let
$$\partial(r)=\sum_{t=1}^ma_t(x_{i_1}\cdots \widehat{x_{i_t}}\cdots x_{i_m})\in S.$$
Then $I(\mathcal{A})=(\partial(r):r\in F(\mathcal{A}))$ is the {\it Orlik-Terao ideal}, and $\mathbf{C}(\mathcal{A})=S/I(\mathcal{A})$ is the {\it Orlik-Terao algebra} of the arrangement $\mathcal{A}$.
\end{definition}

From the similarity between the Orlik-Solomon algebra and the Orlik-Terao algebra it is natural to hope that the Orlik-Terao algebra also encodes useful information about the arrangement (in some sense the Orlik-Terao algebra seems to see ``more" because it records the ``weights" of the dependencies among the hyperplanes). In fact, Orlik-Terao \cite{OT2} proved, when ${K}=\mathbb{R}$, that the dimension of the {\it artinian Orlik-Terao algebra} (i.e., the quotient of $\mathbf{C}(\mathcal{A})$ by the ideal $(x_1^2,\ldots,x_n^2)$) is equal to the number of connected components of the complement $X$ of the hyperplanes. Then Terao \cite{T} computed the Hilbert series of $\mathbf{C}(\mathcal{A})$ via  the Poincar\'{e} polynomial of $\mathcal{A}$ (see Proposition \ref{th26}). In \cite{ST}, Schenck-Tohaneanu raised a new interest in the Orlik-Terao algebra by giving a characterization of the 2-formality of $\mathcal{A}$ in terms of the quadratic component of its Orlik-Terao ideal. See also the survey of Schenck \cite{Sc} for other results and problems concerning the Orlik-Terao algebra.

In this paper we are interested in Orlik-Terao algebras with extremal properties like, e.g., having a linear resolution or being complete intersections. We give characterizations for arrangements whose Orlik-Terao algebra has one of these two properties. It turns out that these properties of the Orlik-Terao algebra are combinatorial, in the sense that they only depend on the underlying matroid of the arrangement.

Our approach is based on a closed connection between the Orlik-Terao ideal and the Stanley-Reisner ideal of the broken circuit complex of the underlying matroid of the arrangement which was in particular studied in \cite{PS}: the latter one is the initial ideal of the former one (see Theorem \ref{th24}). Normally, a property which holds for an ideal need not hold for its initial ideal and vice versa. Fortunately, this is the case for the Orlik-Terao ideal and the two properties we are interested in (see Corollary \ref{co1} and Theorem \ref{th415}). Thus our strategy is as follows: We first consider Stanley-Reisner ideals of the broken circuit complexes of simple matroids and characterize those admitting a linear resolution or being complete intersections (see Theorem \ref{lm1}, Theorem \ref{th41}). These results will then be applied to yield characterizations of arrangements whose Orlik-Terao ideal having the same properties (see Theorem \ref{th34}, Theorem \ref{th415}).

Our results have several interesting consequences. For instance, it is shown in Corollary \ref{co48} that for a triangulation of a simple polygon, its cycle matroid, with respect to a suitable ordering of the edges, has pairwise disjoint minimal broken circuits. Whereas Theorem \ref{th411} is an improvement of Wilf's upper bounds on the coefficients of the chromatic polynomial of a maximal planar graph in \cite{W}. For matroids whose Stanley-Reisner ideal of the broken circuit complex is a complete intersection, we compute the Poincar\'{e} polynomials of their Orlik-Solomon algebras and verify in Theorem \ref{th48} the following conjecture which was first studied in \cite{SY}:
\begin{conjecture}\label{cj}
 A matroid (an arrangement) is supersolvable if and only if its Orlik-Solomon algebra is Koszul.
\end{conjecture}
A similar result for arrangements whose Orlik-Terao algebra is a complete intersection will then be derived in Corollary \ref{th416}. Note that up to now, the above conjecture has been proved for hypersolvable arrangements and graphic arrangements; see \cite{JP}, \cite{SS1}. 

Note also that Denham, Garrousian and Tohaneanu have recently studied Orlik-Terao algebras which are quadratic complete intersections with a different method and they have independently obtained a result similar to Corollary \ref{th416}; see \cite[Corollary 5.12]{DGT}.

Before going into details, let us explain how this paper is organized. In Section 2, we recall some notions and facts concerning broken circuit complexes and hyperplane arrangements. Section 3 is divided into two parts. We first characterize simple matroids whose Stanley-Reisner ideal of the broken circuit complex admits a linear resolution. Characterizations of arrangements whose Orlik-Terao ideal has the same property will be given thereafter. Note that similar characterizations for matroids and hyperplane arrangements whose Orlik-Solomon ideal admits a linear resolution were obtained before in \cite{EPY} and \cite{KR}. Section 4 also contains two parts. In the first part, after giving characterizations of simple matroids whose Stanley-Reisner ideal of the broken circuit complex is a complete intersection, we prove that Conjecture \ref{cj} holds for those matroids. Apart from the applications to graph theory mentioned above, we also show that for  the Stanley-Reisner ideal
 of a broken circuit complex of codimension 3, the Gorensteiness implies the complete intersection property (Proposition \ref{pr412}). In the second part, we characterize arrangements whose Orlik-Terao ideal is a complete intersection and verify again Conjecture \ref{cj} for those arrangements.

\section{Preliminaries}

In this section we review some notions and results from the theory of matroids and hyperplane arrangements which will be used throughout this paper. For unexplained terms and further details we refer to \cite{B}, \cite{OT}, \cite{O} .

Let us first recall the notion of matroid. A {\it matroid} $M$ on the ground set $\Gamma$ is a collection $\mathfrak{I}$ of subsets of $\Gamma$ satisfying the following conditions:
\begin{enumerate}
\item
$\emptyset\in\mathfrak{I};$
\item
If $I\in \mathfrak{I}$ and $I'\subseteq I$, then $I'\in\mathfrak{I}$;
\item
If $I,I'\in \mathfrak{I}$ and $|I'|<| I|$, then there is an element $i\in I-I'$ such that $I'\cup\{i\}\in\mathfrak{I}.$
\end{enumerate}
The members of $\mathfrak{I}$ are called {\it independent sets}. All the maximal independent sets of $M$ have the same cardinality, we call this cardinality the {\it rank} of $M$. {\it Dependent sets } are subsets of $\Gamma$ that are not in $\mathfrak{I}$. Minimal dependent sets are called {\it circuits}. The matroid $M$ is {\it simple} if each circuit has cardinality at least 3. Denote by $\mathfrak{C}(M)$ the set of all circuits of $M$. Clearly,  $\mathfrak{C}(M)$ determines $M$: $\mathfrak{I}$ consists of subsets of $\Gamma$ that do not contain any member of $\mathfrak{C}(M)$.  We will need the following elimination theorem for circuits. A more general version of this result can be found in \cite[Theorem 3]{As}.

\begin{theorem}\label{lm7}
 Let $M$ be a matroid on $\Gamma$ and let $C_1,\ldots,C_m$ be circuits of $M$. Assume that
$$C_i\nsubseteq \bigcup_{j<i}C_j\quad \text{for all}\ i=2,\ldots,m.$$
Then for each subset $B$ of $\Gamma$ with $|B|=m-1$, there exists a circuit $C$ of $M$ such that
$$C\subseteq \bigcup_{j=1}^mC_j-B.$$
\end{theorem}

For two matroids $M_1$ and $M_2$ on disjoint ground set $\Gamma_1$ and $\Gamma_2$, we define their {\it direct sum} $M_1\oplus M_2$ to be the matroid on the ground set $\Gamma_1\cup\Gamma_2$ whose independent sets are the unions of an independent set of $M_1$ and an independent set of $M_2$. In other words, the circuits of $M_1\oplus M_2$ are those of $M_1$ and those of $M_2$, i.e., $\mathfrak{C}(M_1\oplus M_2)=\mathfrak{C}(M_1)\cup\mathfrak{C}(M_2).$

\begin{example}\label{ex1}
(i) Let $m\leq n$ be non-negative integers and let $\Gamma$ be an $n$-element set. The {\it uniform matroid} $U_{m,n}$ on $\Gamma$ is the matroid whose the independent sets are the subsets of $\Gamma$ of cardinality at most $m$. This matroid has rank $m$ and its circuits are the $(m+1)$-element subsets of $\Gamma$. For $m\geq2$, $U_{m,n}$ is simple. When $m=n$, the matroid $U_{n,n}$ has no dependent sets and is called {\it free}.

(ii) Let $\mathcal{A}=\{H_1,\ldots,H_n\}$ be a central hyperplane arrangement in a vector space $V$ and let $\alpha_i\in V^*$ be linear forms such that $\ker \alpha_i=H_i$ for $i=1,\ldots,n.$ Then we can define a matroid $M(\mathcal{A})$ on the ground set $\mathcal{A}$ by taking the independent sets to be the independent subsets of $\mathcal{A}$, i.e., the subsets $\{H_{i_1},\ldots,H_{i_m}\}$ such that $\alpha_{i_1},\ldots,\alpha_{i_m}$ are linearly independent. We call $M(\mathcal{A})$ the {\it underlying matroid} of $\mathcal{A}.$ Clearly, this matroid is simple. In the following we will usually identify the ground set $\mathcal{A}$ with $[n]:=\{1,\ldots,n\}$ and consider $M(\mathcal{A})$ as a matroid on $[n].$

It is apparent that uniform matroids and free matroids are underlying matroids of generic arrangements and Boolean arrangements, respectively. Moreover, if we have two arrangements $\mathcal{A}_1$ and $\mathcal{A}_2$ in vector spaces $V_1$ and $V_2$, then $M(\mathcal{A}_1\times \mathcal{A}_2)=M(\mathcal{A}_1)\oplus M(\mathcal{A}_2)$, where the {\it product arrangement} $\mathcal{A}_1\times \mathcal{A}_2$ is defined in the space $V=V_1\oplus V_2$ as follows:
$$\mathcal{A}_1\times \mathcal{A}_2=\{H_1\oplus V_2: H_1\in\mathcal{A}_1\}\cup\{V_1\oplus H_2: H_2\in\mathcal{A}_2\}.$$

(iii) Let $G$ be a graph whose the edge set is $\mathcal{E}.$ Let $\mathfrak{C}$ be the set of edge sets of cycles of $G$. Then $\mathfrak{C}$ forms the set of circuits of a matroid $M(G)$ on $\mathcal{E}$. We call $M(G)$ the {\it cycle matroid} (or {\it graphic matroid}) of $G$. This matroid is simple if $G$ is a simple graph.
\end{example}

Now assume that $(M,\prec)$ is an ordered matroid of rank $r$ on $[n]$. This means that the matroid $M$ is given with a linear ordering $\prec$ of the ground set $[n]$. (Notice that $\prec$ need not be the ordinary ordering of $[n]$.) For each circuit $C$ of $M$, let $\min_\prec(C)$ be the minimal element of $C$ with respect to $\prec$.  By abuse of notation, we sometimes also write $\min_\prec (C)$ for the set $\{\min_\prec(C)\}$. Then $bc_\prec(C)=C-\min_\prec(C)$ is called a {\it broken circuit}. The {\it broken circuit complex} of $(M,\prec)$, denote by $BC_\prec(M)$, is the collection of all subsets of $[n]$ that do not contain a broken circuit. It is well-known that $BC_\prec(M)$ is an $(r-1)$-dimensional shellable complex; see \cite{Pr} or also \cite[7.4]{B}.
Let ${K}$ be a field and let $\mathcal{I}_\prec(M)\subset S={K}[x_1,\ldots,x_n]$ be the Stanley-Reisner ideal of the broken circuit complex $BC_\prec(M)$. Then $\mathcal{I}_\prec(M)$ is generated by all the monomials $x_{bc_\prec(C)}:=\prod_{i\in bc_\prec(C)}x_i$, where  $C\in \mathfrak{C}(M)$. From the shellability of $BC_\prec(M)$ it follows that the Stanley-Reisner ring $S/\mathcal{I}_\prec(M)$ is a Cohen-Macaulay ring of dimension $r$.

When $M=M(\mathcal{A})$ is the underlying matroid of a central arrangement $\mathcal{A}$, Proudfoot and Speyer \cite{PS} showed that the Stanley-Reisner ring of $BC_\prec(M(\mathcal{A}))$ is a degeneration of the Orlik-Terao algebra of $\mathcal{A}$ for any choice of the ordering $\prec$ (here $M(\mathcal{A})$ is considered as a matroid on $[n]$; see Example \ref{ex1}(ii)). This relation between the two algebras, which plays an important role to our paper, is the spirit of the following theorem. Note that if $C$ is a circuit of $M(\mathcal{A})$, then there exist nonzero scalars $\{a_i:i\in C\}$, unique up to scaling, such that $r_C=\sum_{i\in C}a_ix_i$  is a relation of the relation space $F(\mathcal{A})$. Recall from \cite[Theorem\ 4]{PS}:

\begin{theorem}
\label{th24}
Let $\mathcal{A}$ be a central arrangement of $n$ hyperplanes in a vector space $V$ over a field ${K}$. Let $M=M(\mathcal{A})$ be the underlying matroid of $\mathcal{A}$. Then the set $\{\partial(r_C):C\in \mathfrak{C}(M)\}$ is a universal Gr\"{o}bner basis for the Orlik-Terao ideal $I(\mathcal{A})$ of $\mathcal{A}$. Given any ordering $\prec$ of $[n]$, with an arbitrary induced monomial order on ${K}[x_1,\ldots,x_n]$, we have $\mathrm{in}_\prec(I(\mathcal{A}))=\mathcal{I}_\prec(M).$
\end{theorem}

In particular, it follows from the above theorem that Orlik-Terao ideals are Cohen-Macaulay. These ideals are also prime, as shown in \cite[Proposition 2.1]{ST}.

We now turn to necessary results concerning Orlik-Solomon algebras of matroids. Observe that the definition of the Orlik-Solomon algebra of an arrangement depends only on its underlying matroid and thus can be extended to the matroid level. Let $M$ be a matroid on $[n]$ and let $E={K}\langle e_1,\ldots,e_n\rangle$ be a standard graded exterior algebra over a field ${K}$ (one can also replace ${K}$ by any commutative ring). The {\it Orlik-Solomon ideal} of $M$ is the ideal $J(M)\subset E$ generated by $\partial e_T$ for every dependent set $T$ of $M$. Here, for a subset $T=\{i_1,\ldots,i_m\}$ of $[n]$, we write $e_T=e_{i_1}\cdots e_{i_m}$ and $\partial e_T=\sum_{t=1}^m(-1)^{t-1}e_{T-\{i_t\}}.$  The {\it Orlik-Solomon algebra} $\mathbf{A}(M)$ of $M$ is the quotient ring $E/J(M).$ Assume now that $(M,\prec)$ is an ordered simple matroid of rank $r$. Then we have a decomposition $\mathbf{A}(M)=\bigoplus_{p=0}^r\mathbf{A}_p(M)$ as a graded ${K}$-vector space. Recall
the definition of the {\it Poincar\'{e} polynomial} of $\mathbf{A}(M)$:
$$\pi(\mathbf{A}(M),t)=\sum_{p=0}^r\dim_{K}\mathbf{A}_p(M)t^p.$$
It is known that
\begin{equation}\label{eq21}
\pi(\mathbf{A}(M),t)=\sum_{p=0}^rf_{p-1}t^p,
\end{equation}
where $f_{-1}=1$ and $(f_0,\ldots,f_{r-1})$ is the $f$-vector of the broken circuit complex $BC_\prec(M)$; see \cite[Corollary 7.10.3]{B}. This leads to the following relation between $\pi(\mathbf{A}(M),t)$ and the Hilbert series of the Stanley-Reisner ring of $BC_\prec(M)$, from which a formula of Terao for the Hilbert series of the Orlik-Terao algebra \cite[Theorm 1.2]{T} follows immediately. For a graded ${K}$-vector space $W=\bigoplus_{p\geq0}W_p$, we denote $H_W(t)=\sum_{p\geq0}\dim_{K}W_pt^p$ the Hilbert series of $W.$

\begin{proposition}\label{th26}
Let $(M,\prec)$ be an ordered simple matroid of rank $r$ on $[n]$. Let $\mathcal{I}_\prec(M)\subset S={K}[x_1,\ldots,x_n]$ be the Stanley-Reisner ideal of the broken circuit complex $BC_\prec(M)$. Then we have
$$H_{S/\mathcal{I}_\prec(M)}(t)=\pi\Big(\mathbf{A}(M),\frac{t}{1-t}\Big).$$
In particular, if $M=M(\mathcal{A})$ is the underlying matroid of a central arrangement $\mathcal{A}$ then
$$H_{\mathbf{C}(\mathcal{A})}(t)=H_{S/\mathcal{I}_\prec(M)}(t)=\pi\Big(\mathbf{A}(M),\frac{t}{1-t}\Big).$$
\end{proposition}

\begin{proof}
Let $(f_0,\ldots,f_{r-1})$ be the $f$-vector of the complex $BC_\prec(M)$ and let $f_{-1}=1$. Then it is well-known that
$H_{S/\mathcal{I}_\prec(M)}(t)=\sum_{p=0}^rf_{p-1}\Big(\frac{t}{1-t}\Big)^i$;
see, e.g., \cite[Proposition 6.2.1]{HH}. This, together with (\ref{eq21}), implies the first assertion of the proposition. The second one follows from the first one and Theorem \ref{th24}.
\end{proof}

We conclude this section with a quick review of the chromatic polynomial of a graph. Let $G$ be a simple graph on $\ell$ vertices. For each positive integer $t$, let $\chi(G,t)$ be the number of colorings of $G$ with $t$ colors. This function is a polynomial, called the {\it chromatic polynomial} of $G$. Let $M(G)$ be the cycle matroid of $G$ (see Example \ref{ex1}(iii)) and let $\prec$ be an ordering of the edge set of $G$. A classical theorem of Whitney \cite{Wh} (see also the exposition of Wilf \cite{W}) says that
$$\chi(G,t)=t^\ell-a_1t^{\ell-1}+a_2t^{\ell-2}-\cdots+(-1)^{\ell-1}a_{\ell-1}t,$$
where $(a_1,\ldots,a_r)=(f_0,\ldots,f_{r-1})$ is the $f$-vector of the broken circuit complex $BC_\prec(M(G))$ ($r$ is the rank of $M(G)$) and $a_i=0$ for $i>r.$ By (\ref{eq21}), one can rewrite $\chi(G,t)$ as follows
$$\chi(G,t)=\sum_{p=0}^r(-1)^pf_{p-1}t^{\ell-p}=t^\ell\sum_{p=0}^rf_{p-1}(-t)^{-p}=t^\ell\pi(\mathbf{A}(M(G)),-t^{-1}).$$
Thus we have the well-known result (which is also a consequence of \cite[Corollary 7.10.3]{B}):
\begin{corollary}\label{th27}
Let $G$ be a simple graph on $\ell$ vertices. Then
$\chi(G,t)=t^\ell\pi(\mathbf{A}(M(G)),-t^{-1}).$
\end{corollary}

\section{Cohen-Macaulay ideals and linear resolutions}

Orlik-Solomon ideals admitting a linear free resolution were first characterized by Eisenbud, Popescu and Yuzvinsky \cite[Corollary 3.6]{EPY}. This result was then extended to matroids by K\"{a}mpf and R\"{o}mer \cite[Theorem 6.11]{KR}. In this section, we characterize Orlik-Terao ideals which have a linear resolution. This will be done first for the Stanley-Reisner ideal of the broken circuit complex of a matroid. Our characterizations are similar to those in \cite{EPY}, \cite{KR}.

Recall that $S={K}[x_1,\ldots,x_n]$ is a standard graded polynomial ring over a field ${K}$. Throughout this section, ${K}$ is assume to be infinite. A finitely generated graded $S$-module $W$ is said to have a $p$-{\it linear resolution} if the graded minimal free resolution of $W$ is of the form
$$0\rightarrow S(-p-m)^{\beta_m}\rightarrow\cdots\rightarrow S(-p-1)^{\beta_1}\rightarrow S(-p)^{\beta_0}\rightarrow W\rightarrow0.$$
The following characterization of Cohen-Macaulay ideals with linear resolution is essentially due to Cavaliere, Rossi and Valla \cite[Proposition 2.1]{CRV} (see also Renter\'{i}a and Villarreal \cite[Theorem 3.2]{RV}). We present here another proof for later use.

\begin{proposition}\label{th1}
 Let $I=\bigoplus_{j\geq 0} I_j$ be a graded Cohen-Macaulay ideal in $S$ of codimension $h$. Assume $p$ is the smallest integer such that $I_p\ne 0$. Then the following conditions are equivalent:
\begin{enumerate}
\item
$I$ has a $p$-linear resolution;
\item
For any maximal $S/I$-regular sequence $y_1,\ldots,y_{n-h}$ of linear forms in $S$, we have $\bar I=\overline{\mathfrak{m}}^p$, where $\bar I$ and $\overline{\mathfrak{m}}$ are respectively the image of $I$ and the maximal graded ideal $\mathfrak{m}=(x_1,\ldots,x_n)$ in $S/(y_1,\ldots,y_{n-h})$;
\item
$H(I,p)=\dbinom{p+h-1}{p}$, where $H(I,\lpnt)$ denotes the Hilbert function of $I.$
\end{enumerate}
\end{proposition}

\begin{proof}
Note that there always exists a maximal $S/I$-regular sequence of linear forms in $S$ as the coefficient field ${K}$ is infinite; see, e.g., \cite[Proposition 1.5.12]{BH2}.

(i)$\Rightarrow$(ii): By factoring out the sequence $y_1,\ldots,y_{n-h}$, it is possible to assume that $S/I$ is an artinian ring. Then we have the following formula for the regularity  of this ring:
$$ \mathrm{reg}(S/I)=\max\{i: (S/I)_i\ne 0\};$$
see, e.g., \cite[Theorem 18.4]{P}. On the other hand, since $I$ admits a $p$-linear resolution,  it is well-known that
$$\mathrm{reg}(S/I)=\mathrm{reg}(I)-1=p-1;$$
see, e.g., \cite[Proposition 18.2]{P}. Thus we obtain $\max\{i: (S/I)_i\ne 0\}=p-1$, which simply means that $I=\mathfrak{m}^p.$

(ii)$\Rightarrow$(i): Since $\overline{\mathfrak{m}}^p$ has linear quotients, it admits a $p$-linear resolution; see, e.g., \cite[Proposition 8.2.1]{HH}. It follows that $\bar I$, and thus $I$, also admits a $p$-linear resolution.

(ii)$\Leftrightarrow$(iii): Note that if a linear form $y\in S$ is a nonzero divisor on $S/I$ then
$$\frac{I+(y)}{(y)}\cong \frac{I}{(y)\cap I}=\frac{I}{y I}=\frac{\bigoplus_{j\geq p} I_j}{\bigoplus_{j\geq p} yI_j}.$$
In particular, $H(I,p)=H\Big(\frac{I+(y)}{(y)},p\Big)$. Now since $y_1,\ldots,y_{n-h}$ is an $S/I$-regular sequence and $\bar I\subseteq\overline{\mathfrak{m}}^p$ we have
$$
H(I,p)=H(\bar I,p)\le H(\overline{\mathfrak{m}}^p,p)=\dbinom{p+h-1}{p},
$$
with equality if and only if $\bar I=\overline{\mathfrak{m}}^p.$ Note that the last equality in the above equation follows from the fact that $S/(y_1,\ldots,y_{n-h})$ is a polynomial ring in $h$ variables over ${K}$.
\end{proof}

\begin{corollary}\label{co1}
Let $I$ be a graded ideal in $S$ and $\prec$ a monomial order on $S$. Assume that $\mathrm{in}_\prec(I)$ is Cohen-Macaulay. Then $I$ has a linear resolution if and only if $\mathrm{in}_\prec(I)$ has one.
\end{corollary}

\begin{proof}
Note that $I$ is also a Cohen-Macaulay ideal (see, e.g., \cite[Corollary 3.3.5]{HH}). The corollary now follows from the equivalent of conditions (i) and (iii) in Proposition \ref{th1} since $I$ and $\mathrm{in}_\prec(I)$ have the same codimension and Hilbert function.
\end{proof}

\subsection{Stanley-Reisner ideals of broken circuit complexes}

Let $(M,\prec)$ be an ordered simple matroid of rank $r$ on $[n]$. Let $\mathcal{I}_\prec(M)\subset S$ be the Stanley-Reisner ideal of the broken circuit complex $BC_\prec(M)$. Those matroids $M$ whose $\mathcal{I}_\prec(M)$ admits a linear resolution are characterized in the following theorem. For Orlik-Solomon ideals, a similar characterization can be found in \cite[Theorem 6.11]{KR}. In fact, one can prove the theorem by utilizing \cite[Theorem 6.11]{KR} and \cite[Corollary 2.2]{AAH}. However, we present here a somewhat more direct proof which does not involve exterior algebras.

\begin{theorem}\label{lm1}
Let $(M,\prec)$ be an ordered simple matroid of rank $r$ on $[n]$ and let $\mathcal{I}_\prec(M)$ be the Stanley-Reisner ideal of the broken circuit complex of $M$. Then the following conditions are equivalent:
\begin{enumerate}
\item $\mathcal{I}_\prec(M)$ has a $p$-linear resolution;
\item $2\leq p\leq r$ and $M$ is isomorphic to $U_{p,n-r+p}\oplus U_{r-p,r-p}$.
\end{enumerate}
\end{theorem}

\begin{proof}
(ii)$\Rightarrow$(i): Assume $M$ is isomorphic to $U_{p,n-r+p}\oplus U_{r-p,r-p}$. Then after renumbering the variables (if necessary) we get
$$\mathcal{I}_\prec(M)=(x_{i_1}\cdots x_{i_p}: 1\le i_1<\cdots< i_p< n-r+p).$$
This ideal clearly has linear quotients, and consequently, it has a linear resolution.

(i)$\Rightarrow$(ii): Assume $\mathcal{I}_\prec(M)$ has a $p$-linear resolution. Evidently, $2\leq p\leq r$ as $M$ is simple. Recall that the ring $S/\mathcal{I}_\prec(M)$ is Cohen-Macaulay of dimension $r.$ Let ${\bf y}=y_1,\ldots,y_r$ be a maximal $S/\mathcal{I}_\prec(M)$-regular sequence of linear forms in $S$. Denote by $\bar{\mathcal{I}}_\prec(M)$ the image of $\mathcal{I}_\prec(M)$ in $\bar{S}=S/(\bf{y})$. It follows from Proposition \ref{th1} that
$$R:=S/\big(\mathcal{I}_\prec(M)+({\bf y})\big)\cong\bar{S}/\bar{\mathcal{I}}_\prec(M)=\bar{S}/\overline{\mathfrak{m}}^p\cong \frac{{K}[z_1,\ldots,z_{n-r}]}{(z_1,\ldots,z_{n-r})^p},$$
where $\mathfrak{m}=(x_1,\ldots,x_n)$ and $z_1,\ldots,z_{n-r}$ are variables. Since $\bf y$ is an $S/\mathcal{I}_\prec(M)$-sequence, one gets the following relation between the Hilbert series of $S/\mathcal{I}_\prec(M)$ and $R$:
$$H_{S/\mathcal{I}_\prec(M)}(t)=H_R(t)/(1-t)^r.$$
The $h$-vector $(h_0,\ldots,h_r)$ of $S/\mathcal{I}_\prec(M)$ is now computable:
$$h_k=H(R,k)=
\begin{cases}
\binom{n-r+k-1}{k}&\ \text{for}\ 0\le k\le p-1\\
0 &\ \text{for}\ p\le k\le r,
\end{cases}$$
where $H(R,\lpnt)$ denotes the Hilbert function of $R$. This yields the following formula for the $f$-vector $(f_0,\ldots,f_{r-1})$ of $S/\mathcal{I}_\prec(M)$:
$$
f_{k-1}=\sum_{i=0}^rh_i\binom{r-i}{k-i}=\sum_{i=0}^{p-1}\binom{n-r+i-1}{i}\binom{r-i}{k-i} \ \text{for}\ k=0,\ldots,r.
$$
Note that $c=p+1$ is the smallest size of a circuit of $M$ since $\mathcal{I}_\prec(M)$ is generated by monomials of degree $p$. So by \cite[Proposition 7.5.6]{B}, the $f$-vector of $S/\mathcal{I}_\prec(M)$ attains its minimum and this forces  $M$ to be isomorphic to $U_{p,n-r+p}\oplus U_{r-p,r-p}$.
\end{proof}

\begin{corollary}
With the assumption of Theorem \ref{lm1}, if $\mathcal{I}_\prec(M)$ has a linear resolution, then so do all of its powers.
\end{corollary}

\begin{proof}
By Theorem \ref{lm1}, if $\mathcal{I}_\prec(M)$ has a linear resolution, then it is a so-called squarefree Veronese ideal. It is known that all powers of this ideal have linear quotients; see \cite[Corollary 12.6.4]{HH}. Therefore, they all have a linear resolution.
\end{proof}

\subsection{Orlik-Terao ideals}

Return to our assumption in the introduction: $\mathcal{A}$ is an essential central arrangement of $n$ hyperplanes in an $r$-dimensional vector space $V$ over ${K}$. Let $M(\mathcal{A})$ be the underlying matroid and $I(\mathcal{A})\subset S$ the Orlik-Terao ideal of $\mathcal{A}$. We refer to \cite[Definition 1.15]{OT} for the coning construction of an arrangement. Characterizations of arrangements whose Orlik-Terao ideal has a linear resolution are given below. It turns out that this property of the Orlik-Terao ideal is combinatorial and holds for ``almost all" arrangements.

\begin{theorem}\label{th34}
 For an essential central arrangement $\mathcal{A}$ of $n$ hyperplanes in a vector space of dimension $r$, the following conditions are equivalent:
\begin{enumerate}
\item $I(\mathcal{A})$ has a $p$-linear resolution;
\item $2\leq p\leq r$ and $M(\mathcal{A})$ is isomorphic to $U_{p,n-r+p}\oplus U_{r-p,r-p}$;
\item $2\leq p\leq r$ and $\mathcal{A}=\mathcal{A}_1\times \mathcal{A}_2$, where $\mathcal{A}_1$ is a generic central arrangement of $n-r+p$ hyperplanes in a $p$-dimensional vector space and $\mathcal{A}_2$ is a Boolean arrangement in an $(r-p)$-dimensional vector space;
\item $2\leq p\leq r$ and $\mathcal{A}$ is obtained by successively coning a generic central arrangement of $n-r+p$ hyperplanes in a $p$-dimensional vector space.
\end{enumerate}
\end{theorem}

\begin{proof}
For an ordering  $\prec$ of the ground set $[n]$ of $M(\mathcal{A})$, we use the same notation to denote an induced monomial order on $S$. Then by Theorem \ref{th24}, $\mathrm{in}_\prec(I(\mathcal{A}))=\mathcal{I}_\prec\big(M(\mathcal{A})\big).$ Now the equivalence of (i) and (ii) follows by combining Corollary \ref{co1} and Theorem \ref{lm1}. Whereas the equivalences of (ii) and (iii), (iii) and (iv) are just a matter of interpreting terminologies.
\end{proof}

Before going further, let us recall shortly here the notion of Koszul algebra. For more information, we refer to the survey of Fr\"{o}berg \cite{F}. Let $B=\mathcal{S}/I$  be a graded ${K}$-algebra, where $\mathcal{S}$ is either a polynomial algebra or an exterior algebra over ${K}$ and $I$ is a graded ideal of $\mathcal{S}$. Then $B$ is called a {\it Koszul algebra} if ${K}$ has a linear resolution over $B$. It is well-known that if $B$ is Koszul then $I$ is generated by quadrics. The converse is not true in general. However, it follows from a result of Fr\"{o}berg that if $I$ has a quadratic Gr\"{o}bner basis then $B$ is Koszul.

The following consequence is immediate from the above theorem.

\begin{corollary}
 Let $\mathcal{A}$ be an essential central arrangement. Then $I(\mathcal{A})$ has a $2$-linear resolution if and only if $\mathcal{A}$ is obtained by successively coning a central arrangement of lines in a plane. In this case, the Orlik-Terao algebra $\mathbf{C}(\mathcal{A})$ is Koszul.
\end{corollary}

\section{The complete intersection property}

The broken circuit complex was introduced by Wilf in \cite{W}. There he found several necessary conditions for a polynomial to be the chromatic polynomial of a graph. He also computed the chromatic polynomials of the graphs that admit a broken circuit complex with disjoint minimal broken circuits, and derived from that upper bounds for coefficients of the chromatic polynomial of a maximal planar graph. In this section, we characterize, in terms of the set of circuits, those ordered matroids whose minimal broken circuits are pairwise disjoint, i.e., those ordered matroids whose Stanley-Reisner ideal of the broken circuit complex is a complete intersection. This result is applied to triangulations of  simple polygons to show that the cycle matroid of such a graph admits a broken circuit complex with disjoint minimal broken circuits. Then we show that Conjecture \ref{cj} holds for matroids  whose minimal broken circuits are pairwise disjoint. As another application, we improve Wilf's upper bounds mentioned above. We also show, in codimension 3, that Gorensteiness of the Stanley-Reisner ideal of the broken circuit complex is equivalent to be a complete intersection. Finally, we characterize arrangements whose Orlik-Terao ideal is a complete intersection and verify Conjecture \ref{cj} for those arrangements. For the last result see also \cite[Cor. 5.12]{DGT} who proved independently a variation of this  statement with a different method.

\subsection{Stanley-Reisner ideals of broken circuit complexes}

Let $(M,\prec)$ be an ordered simple matroid on $[n]$. We keep some notation introduced before:  $\mathfrak{C}(M)$ is the set of circuits of $M$; $\mathcal{I}_\prec(M)\subset S={K}[x_1,\ldots,x_n]$ denotes the Stanley-Reisner ideal of the broken circuit complex $BC_\prec(M)$; and $\min_\prec (C)$ and $bc_\prec(C)$ are respectively the minimal element and the broken circuit of a given circuit $C$ with respect to $\prec$. Recall that $\mathcal{I}_\prec(M)=(x_{bc_\prec(C)}:C\in \mathfrak{C}(M))$, where  $x_{bc_\prec(C)}=\prod_{i\in bc_\prec(C)}x_i$.

Let $\mathfrak{D}$ be a subset of $\mathfrak{C}(M)$. We call $\mathfrak{D}$ a {\it generating set} of $\mathfrak{C}(M)$ if $\{x_{bc_\prec(C)}: C\in \mathfrak{D}\}$ generates $\mathcal{I}_\prec(M)$. Obviously, $\mathfrak{D}$ is a generating set of $\mathfrak{C}(M)$ if and only if for any $C'\in \mathfrak{C}(M)$, there is a $C\in \mathfrak{D}$ with $bc_\prec(C)\subseteq bc_\prec(C')$, or, in other words, $\{bc_\prec(C): C\in\mathfrak{D}\}$ contains the set of minimal broken circuits of $M$.

Let $\mathcal{G}(\mathfrak{D})$ be the {\it intersection graph} of  $\mathfrak{D}$, i.e., the graph whose vertex set is $\mathfrak{D}$ and edges are pairs $\{C,C'\}$ with $C\cap C'\ne \emptyset$. We say that $\mathfrak{D}$ is {\it connected} (respectively, a {\it tree}, a {\it forest}) when so is the graph $\mathcal{G}(\mathfrak{D})$.

We will often consider those subsets $\mathfrak{D}$ of $\mathfrak{C}(M)$ with this property: for any distinct elements $C,C'\in\mathfrak{D}$, one has either $C\cap C'=\min_\prec(C)$ or $C\cap C'=\min_\prec(C')$ whenever $C\cap C'\ne \emptyset$. We call them {\it simple} subsets. Apparently, $\mathfrak{D}$ is a simple subset of $\mathfrak{C}(M)$ if and only if the broken circuits of the elements of $\mathfrak{D}$ are pairwise disjoint.

Now for each subset $\mathfrak{D}$ of  $\mathfrak{C}(M)$, set
$$\mathcal{C}(\mathfrak{D})=\bigcup_{D\in\mathfrak{D}}D-\bigcup_{D,D'\in\mathfrak{D},D\ne D'}(D\cap D') =\bigcup_{D\in\mathfrak{D}}\big(D-\bigcup_{D'\in\mathfrak{D}-\{D\}}(D\cap D')\big) .$$
Then our characterizations for the complete intersection property of the ideal $\mathcal{I}_\prec(M)$ can be stated as follows.

\begin{theorem}\label{th41}
Let $(M,\prec)$ be an ordered simple matroid on $[n]$. The following conditions are equivalent:
\begin{enumerate}
\item $\mathcal{I}_\prec(M)$ is a complete intersection;
\item The minimal broken circuits of $M$ are pairwise disjoint;
\item There exists a simple subset $\mathfrak{D}$ of $\mathfrak{C}$ such that
$$\mathfrak{C}(M)=\{\mathcal{C}(\mathfrak{D}'): \mathfrak{D}' \subseteq\mathfrak{D}\text{ is a tree} \}.$$
\end{enumerate}
\end{theorem}

To prove this theorem, we need some preparations.

\begin{lemma}\label{lm6}
Let $\mathfrak{D}\subseteq\mathfrak{C}(M)$ be a simple subset of cardinality $m$. Then the following statements hold.
\begin{enumerate}[\rm(i)]
\item
There is an enumeration of elements of $\mathfrak{D}$, say as $C_1,\ldots,C_m$, such that
$$\big|C_i\cap\big(\bigcup_{j<i}C_j\big)\big|\leq 1\quad \text{for all} \ i=2,\ldots,m.$$
Moreover, if $\mathfrak{D}'\subseteq \mathfrak{D}$ is connected, then there exists such an enumeration so that the elements of $\mathfrak{D}'$ appear first.
\item
$\mathfrak{D}$ is a tree if and only if $\mathfrak{D}$ is connected and any three distinct elements of $\mathfrak{D}$ have empty intersection.
\item We have
$$\big|\bigcup_{C,C'\in\mathfrak{D},C\ne C'}(C\cap C')\big|\leq m-1,$$
 with equality if and only if $\mathfrak{D}$ is a tree.
\end{enumerate}
\end{lemma}

\begin{proof}
(i) By induction on $m$, to prove the first assertion it suffices to show that there exists a circuit $C\in\mathfrak{D}$ such that $$d(C):=\big|C\cap\big(\bigcup_{C'\in\mathfrak{D}-\{C\}}C'\big)\big|\leq 1.$$
Assume the contrary, i.e., $d(C)\geq 2$ for all $C\in\mathfrak{D}$. Consider the intersection graph $\mathcal{G}(\mathfrak{D})$ of $\mathfrak{D}$. For each edge $\{C,C'\}$ of $\mathcal{G}(\mathfrak{D})$ we call $C\cap C'$ its label. Then it is easily seen that $\mathcal{G}(\mathfrak{D})$ contains a cycle $C_1\ldots C_k$ with pairwise distinct edge labels, i.e., $C_i\cap C_{i+1}\ne C_{j}\cap C_{j+1}$ for $i\ne j$ ($C_{k+1}=C_1$). Let $\{e_i\}=C_i\cap C_{i+1}$ and assume  $e_1=\min\{e_i:i=1,\ldots,k\}$. Recall that one has either $e_1=\min_\prec(C_1)$ or $e_1=\min_\prec(C_2)$. We will consider the case $e_1=\min_\prec(C_1)$, the other one can be treated similarly. Since
$$\{e_k\}=C_k\cap C_1\ne C_1\cap C_2=\{e_1\},$$
it follows that $e_k=\min_\prec(C_k)$. Proceeding in this way, we obtain $e_i=\min_\prec(C_i)$ for all $i=1,\ldots,k.$ In particular, we have $e_2=\min_\prec(C_2)\leq e_1\in C_2$. This, however, is impossible because $e_1\ne e_2$ and $e_1=\min\{e_i\}\leq e_2$.

In order to prove the second assertion, we first enumerate the set $\mathfrak{D}'$ as in the first assertion and then try to enumerate the set $\mathfrak{D}-\mathfrak{D}'$ to get a desired enumeration of $\mathfrak{D}$. The case that $C\cap\big(\bigcup_{C'\in\mathfrak{D}'}C'\big)=\emptyset$ for all $C\in\mathfrak{D}-\mathfrak{D}'$ is trivial: any enumeration of $\mathfrak{D}-\mathfrak{D}'$ as in the first assertion works. In the remaining case choose $C\in\mathfrak{D}$ such that $C\cap\big(\bigcup_{C'\in\mathfrak{D}'}C'\big)\ne\emptyset.$ If we can show that $|C\cap\big(\bigcup_{C'\in\mathfrak{D}'}C'\big)|=1$ then the assertion will follow by induction. Assume that there are $C_1,C_2\in\mathfrak{D}'$ such that $C\cap C_1\ne C\cap C_2$. Since $\mathcal{G}(\mathfrak{D}')$ is connected, there exists a path in $\mathcal{G}(\mathfrak{D}')$ connecting $C_1$ and $C_2$. It follows that $\mathcal{G}(\mathfrak{D})$ has cycles containing $\{C_1,C,C_2\}$. Let $\gamma$ be such a cycle wit
 h shortest length. Then it is easy to see that the labels of the edges of $\gamma$ are pairwise distinct. But this cannot be the case as we have shown before.

(ii) If three distinct elements $C_1,C_2,C_3$ of $\mathfrak{D}$ have non-empty intersection, then they form a cycle in the graph $\mathcal{G}(\mathfrak{D})$, hence $\mathfrak{D}$ cannot be a tree. Conversely, assume that $\mathfrak{D}$ is connected. If $\mathfrak{D}$ is not a tree, then $\mathcal{G}(\mathfrak{D})$ must contain some cycle $\gamma$. As shown in (i), there are two edges of $\gamma$ which share the same label. The vertices of these two edges then have non-empty intersection.

(iii) Enumerate the elements of $\mathfrak{D}$ as in (i). We have
$$\begin{aligned}
\big|\bigcup_{C,C'\in\mathfrak{D},C\ne C'}(C\cap C')\big|&=\big|\bigcup_{i=2}^m\big(C_i\cap\big(\bigcup_{j<i}C_j\big)\big)\big|
\leq\sum_{i=2}^m\big|C_i\cap\big(\bigcup_{j<i}C_j\big)\big|\leq m-1.
\end{aligned}$$
The equality holds if and only if the sets $C_i\cap\big(\bigcup_{j<i}C_j\big)$ for $i=2,\ldots,m$ satisfy two conditions:
they are non-empty; and, they are pairwise distinct. Observe that the first condition is equivalent to the connectedness of $\mathfrak{D}$, while the second one means that the intersection of any three distinct elements of $\mathfrak{D}$ is empty.  The assertion now follows from (ii).
\end{proof}

\begin{remark}\label{rm43}
The proof of Lemma \ref{lm6}(i) is based on a fact that the graph $\mathcal{G}(\mathfrak{D})$ contains no cycles whose edges have pairwise distinct labels. So when $\mathcal{G}(\mathfrak{D})$ has no cycles with pairwise distinct edge labels (in particular, when $\mathcal{G}(\mathfrak{D})$ has no cycles at all, i.e., $\mathfrak{D}$ is a forest) and any two distinct members of $\mathfrak{D}$ intersect in at most one element (but $\mathfrak{D}$ need not be simple), the conclusion of Lemma \ref{lm6}(i) is still true. Moreover, in this case, there is an ordering of the ground set $[n]$ such that $\mathfrak{D}$ is simple with respect to this ordering. Indeed, one first enumerates the elements of $\mathfrak{D}$ as $C_1,\ldots,C_m$ such that
$$\big|C_i\cap\big(\bigcup_{j<i}C_j\big)\big|\leq 1\quad \text{for all} \ i=2,\ldots,m.$$
Set $D_1=C_1$ and $D_i=C_i-\bigcup_{j<i}C_j$ for $ i=2,\ldots,m.$ Then any ordering $\prec$ of $[n]$ such that $d_i\prec d_j$ whenever $d_i\in D_i,d_j\in D_j$ and $i<j$ satisfies the requirement.
\end{remark}

\begin{lemma}\label{lm8}
Let $\mathfrak{D},\mathfrak{D}'\subseteq\mathfrak{C}(M)$ be non-empty simple subsets. Then the following statements hold.
\begin{enumerate}[\rm(i)]
\item
There exists a circuit $C\in \mathfrak{C}(M)$ such that $C\subseteq\mathcal{C}(\mathfrak{D}).$
\item
If $\mathcal{C}(\mathfrak{D})\subseteq C'$ for some circuit $C'\in \mathfrak{C}(M)$, then $\mathcal{C}(\mathfrak{D})=C'$ and $\mathfrak{D}$ is a tree.
\item
If $\mathfrak{D}$ is a tree, $\mathfrak{D}\cup\mathfrak{D}'$ is simple and $\mathcal{C}(\mathfrak{D}')\subseteq\mathcal{C}(\mathfrak{D})$, then $\mathfrak{D}=\mathfrak{D}'$.
\end{enumerate}
\end{lemma}

\begin{proof}
Enumerate the elements of $\mathfrak{D}$ as in Lemma \ref{lm6}(i). It is clear that with this enumeration we have
$$C_i\nsubseteq \bigcup_{j<i}C_j\quad \text{for all}\ i=2,\ldots,m,$$
where $m=|\mathfrak{D}|.$ By virtue of  Lemma \ref{lm6}(iii), one can choose a subset $B$ of $[n]$ with $|B|=m-1$ so that $B$ contains $\bigcup_{D,D'\in\mathfrak{D},D\ne D'}(D\cap D').$ It now follows from Theorem \ref{lm7} that there exists a circuit $C\in\mathfrak{C}(M)$ such that
$$C\subseteq \bigcup_{i=1}^mC_i-B\subseteq\bigcup_{D\in\mathfrak{D}}D-\bigcup_{D,D'\in\mathfrak{D},D\ne D'}(D\cap D') =\mathcal{C}(\mathfrak{D}).$$

If there is another circuit $C'\in\mathfrak{C}(M)$ with $\mathcal{C}(\mathfrak{D})\subseteq C'$, then since $C\subseteq C'$ are both circuits we must have $C=C'$. This implies that $C'=\mathcal{C}(\mathfrak{D})$ and $B=\bigcup_{D,D'\in\mathfrak{D},D\ne D'}(D\cap D').$ As
$$|\bigcup_{D,D'\in\mathfrak{D},D\ne D'}(D\cap D')|=|B|=m-1,$$ Lemma \ref{lm6}(iii) guarantees that $\mathfrak{D}$ is a tree.

To prove (iii), we first show that $\mathfrak{D}'\subseteq\mathfrak{D}$. Indeed, we have
$$\mathcal{C}(\mathfrak{D})=\mathcal{C}(\mathfrak{D})\cup \mathcal{C}(\mathfrak{D}')\supseteq\mathcal{C}(\mathfrak{D}\cup\mathfrak{D}').$$
Choose an enumeration of elements of $\mathfrak{D}\cup\mathfrak{D}'$ in which the elements of $\mathfrak{D}$ appear first as in Lemma \ref{lm6}(i). If $\mathfrak{D}'\nsubseteq\mathfrak{D}$, then there exists a circuit $D'\in\mathfrak{D}'$ (for instance, $D'$ can be chosen to be the last element in the enumeration) such that
$$D'\nsubseteq\bigcup_{D'\ne D\in\mathfrak{D}\cup\mathfrak{D}'}D.$$
Then for any $d\in D'-\bigcup_{D'\ne D\in\mathfrak{D}\cup\mathfrak{D}'}D$ we have $d\in\mathcal{C}(\mathfrak{D}\cup\mathfrak{D}')-\mathcal{C}(\mathfrak{D}).$ This contradiction shows that $\mathfrak{D}'\subseteq\mathfrak{D}$. Suppose $\mathfrak{D}'\ne\mathfrak{D}$. Then since $\mathfrak{D}$ is connected, there exist $D_1\in\mathfrak{D}'$ and $D_2\in\mathfrak{D}-\mathfrak{D}'$ such that $D_1\cap D_2\ne\emptyset$. The fact that three distinct elements of $\mathfrak{D}$ have empty intersection (see Lemma \ref{lm6}(ii)) yields
$$D_1\cap D_2\nsubseteq\bigcup_{D,D'\in\mathfrak{D}',D\ne D'}(D\cap D').$$
This implies $D_1\cap D_2\subseteq\mathcal{C}(\mathfrak{D}')-\mathcal{C}(\mathfrak{D})$, which contradicts the hypothesis. Hence $\mathfrak{D}=\mathfrak{D}'$.
\end{proof}

\begin{lemma}\label{lm9}
Let $\mathfrak{D}\subseteq\mathfrak{C}(M)$ be simple. Assume that $\mathfrak{D}$ is a generating set of $\mathfrak{C}(M)$. Then for any $C\in\mathfrak{C}(M)$, there exists a subset $\mathfrak{D}'\subseteq \mathfrak{D}$ such that $C\subseteq \bigcup_{D\in\mathfrak{D}'}D$ and
$$D\subseteq C\cup \big(\bigcup_{D'\in\mathfrak{D}'-\{D\}} D'\big)\quad \text{for all } D\in\mathfrak{D}'.$$
If this is the case, then $\mathfrak{D}'$ must be a tree and we have
$C=\mathcal{C}(\mathfrak{D}').$
\end{lemma}

\begin{proof}
Suppose on the contrary that there is no such $\mathfrak{D}'.$ If $C\subseteq \bigcup_{D\in\mathfrak{D}}D$ (we will not exclude the case that $C\nsubseteq \bigcup_{D\in\mathfrak{D}}D$ in our argument below), then there must be some $D_m\in\mathfrak{D}$ ($m=|\mathfrak{D}|$) such that
$$D_m\nsubseteq C\cup \big(\bigcup_{D\in\mathfrak{D}_m}D\big),$$
where $\mathfrak{D}_m=\mathfrak{D}-\{D_m\}.$ If $C\subseteq\bigcup_{D\in\mathfrak{D}_m}D$, then there exists again some $D_{m-1}\in\mathfrak{D}_m$ such that
$$D_{m-1}\nsubseteq C\cup \big(\bigcup_{D\in\mathfrak{D}_{m-1}}D\big),$$
where $\mathfrak{D}_{m-1}=\mathfrak{D}-\{D_{m-1},D_m\}.$ Continuing this argument, we eventually get an index $1\leq i\leq m+1$, elements $D_i,\ldots,D_m\in\mathfrak{D}$, and subsets $\mathfrak{D}_j=\mathfrak{D}-\{D_{j},\ldots,D_m\}$ ($j=i,\ldots,m$) such that
$$
\begin{aligned}
C&\nsubseteq \bigcup_{D\in\mathfrak{D}_i}D,\\
D_j&\nsubseteq C\cup \big(\bigcup_{D\in\mathfrak{D}_j}D\big)\ \text{ for all }\ j=i,\ldots,m.
\end{aligned}
$$
(The case $i=m+1$ simply means that $C\nsubseteq \bigcup_{D\in\mathfrak{D}}D$.) Now enumerate the set $\mathfrak{D}_i$ as in Lemma \ref{lm6}(i): $\mathfrak{D}_i=\{D_1,\ldots,D_{i-1}\}.$ Then for $j=2,\ldots,i-1$ we have
$$D_j\nsubseteq \bigcup_{l<j}D_l .$$
So the following enumeration of the set $\mathfrak{D}\cup\{C\}$:
$$D_1,\ldots,D_{i-1},C,D_{i},\ldots,D_m$$
satisfies the hypothesis of Theorem \ref{lm7}. Therefore, if we take $d_j\in bc_\prec(D_j)$ for $j=1,\ldots,m$, then there exists a circuit $C'\in\mathfrak{C}(M)$ such that
$$C'\subseteq C\cup \big( \bigcup_{j=1}^mD_j\big)-\{d_1,\ldots,d_m\},$$
Obviously, $bc_\prec(D_j)\nsubseteq bc_\prec(C')$ for all $j=1,\ldots,m.$ But this contradicts the hypothesis that $\mathfrak{D}$ is a generating set of $\mathfrak{C}(M)$. Hence, there must be a subset $\mathfrak{D}'$ of $\mathfrak{D}$ having the required properties.

Since $D\subseteq C\cup \big(\bigcup_{D'\in\mathfrak{D}'-\{D\}}D'\big)$ we have
$$D-\bigcup_{D'\in\mathfrak{D}'-\{D\}}(D\cap D')\subseteq C\quad \text{for all } D\in\mathfrak{D}'.$$
It follows that
$$\mathcal{C}(\mathfrak{D}')=\bigcup_{D\in\mathfrak{D}}\big(D-\bigcup_{D'\in\mathfrak{D}-\{D\}}(D\cap D')\big)\subseteq C.$$
The last assertion now follows from Lemma \ref{lm8}(ii).
\end{proof}

Let $\mathfrak{D}$ be a simple subset of $\mathfrak{C}(M)$ which is also a tree. We have not yet known whether $\mathcal{C}(\mathfrak{D})$ is a circuit of $M$ (this is true, though, at least in the case where the minimal broken circuits of $M$ are pairwise disjoint, as will be proved below). However, in the following lemma we still use the notation $bc_\prec(\mathcal{C}(\mathfrak{D}))$ to denote the set $\mathcal{C}(\mathfrak{D})-\min_\prec(\mathcal{C}(\mathfrak{D})).$

\begin{lemma}\label{lm10}
Let $\mathfrak{D}\subseteq\mathfrak{C}(M)$ be a non-empty simple subset. Assume further that $\mathfrak{D}$ is a tree. Then there exists a circuit $C\in \mathfrak{D}$ such that $bc_\prec(C)\subseteq bc_\prec(\mathcal{C}(\mathfrak{D})).$
\end{lemma}

\begin{proof}
The case $|\mathfrak{D}|=1$ is trivial, so we will assume that $|\mathfrak{D}|\geq2$. Then it is a basic fact in graph theory that the tree $\mathcal{G}(\mathfrak{D})$ has at least two leaves; see, e.g., \cite[Proposition 4.2]{BM}. Thus there are two circuits $C_1,C_2\in\mathfrak{D}$ such that
$$\big|C_i\cap\big(\bigcup_{C\in\mathfrak{D}-\{C_i\}}C\big)\big|= 1\quad \text{for } i=1,2.$$
Denote by $\mathrm{m}(\mathfrak{D})$ the set $\{\min_\prec(C):C\in\mathfrak{D}\}$. Recall that $$\bigcup_{C,C'\in\mathfrak{D},C\ne C'}(C\cap C')\subseteq \mathrm{m}(\mathfrak{D}),$$
and since $\mathfrak{D}$ is a tree,
$$\big|\bigcup_{C,C'\in\mathfrak{D},C\ne C'}(C\cap C')\big|=|\mathfrak{D}|-1\geq |\mathrm{m}(\mathfrak{D})|-1.$$
Thus there might be at most one element of $\mathrm{m}(\mathfrak{D})$ which is not in $\bigcup_{C,C'\in\mathfrak{D},C\ne C'}(C\cap C').$
It follows that $C_i\cap\big(\bigcup_{C\in\mathfrak{D}-\{C_i\}}C\big)={\min_\prec(C_i)}$ for either $i=1$ or $i=2.$ Let us assume, say, that $i=1$. Then
$$bc_\prec(C_1)=C_1-\mathrm{min}_\prec(C_1)=C_1-\bigcup_{C\in\mathfrak{D}-\{C_1\}}(C_1\cap C)\subseteq \mathcal{C}(\mathfrak{D}).$$
Consider the following two cases:

Case 1: $C_2\cap\big(\bigcup_{C\in\mathfrak{D}-\{C_2\}}C\big)={\min_\prec(C_2)}$. Then we also have
$bc_\prec(C_2)\subseteq \mathcal{C}(\mathfrak{D})$ as above.
 Since
$bc_\prec(C_1)\cap bc_\prec(C_2)=\emptyset,$ it follows that $\min_\prec (\mathcal{C}(\mathfrak{D}))\not\in bc_\prec(C_1)$ or $\min_\prec(\mathcal{C}(\mathfrak{D}))\not\in bc_\prec(C_2)$. Hence, we get either $bc_\prec(C_1)\subseteq bc_\prec(\mathcal{C}(\mathfrak{D}))$ or $bc_\prec(C_2)\subseteq bc_\prec(\mathcal{C}(\mathfrak{D})).$

Case 2: $C_2\cap\big(\bigcup_{C\in\mathfrak{D}-\{C_2\}}C\big)\ne{\min_\prec(C_2)}$. Then $\min_\prec(C_2)\in\mathcal{C}(\mathfrak{D})$. Let $D_1D_2\ldots D_s$ be a path in the intersection graph $\mathcal{G}(\mathfrak{D})$ which connects $C_1$ and $C_2$ ($C_1=D_1,C_2=D_s$). Note that $\min_\prec(D_s)\ne D_{s-1}\cap D_s$ because $\min_\prec(D_s)=\min_\prec(C_2)\in\mathcal{C}(\mathfrak{D})$. Hence, $D_{s-1}\cap D_s={\min_\prec(D_{s-1})}$. Consequently, $D_{i}\cap D_{i+1}={\min_\prec(D_{i})}$ for $i=1,\ldots,s-1$ since any three distinct elements of $\mathfrak{D}$ have empty intersection, by Lemma \ref{lm6}(ii). From this we get
$$\mathrm{min}_\prec(C_2)=\mathrm{min}_\prec(D_s)< \mathrm{min}_\prec(D_{s-1})<\cdots< \mathrm{min}_\prec(D_1)=\mathrm{min}_\prec(C_1).$$
Thus $\min_\prec(\mathcal{C}(\mathfrak{D}))$, which is not greater than $\min_\prec(C_2)$, does not belong to $bc_\prec(C_1)$. This yields $bc_\prec(C_1)\subseteq bc_\prec(\mathcal{C}(\mathfrak{D}))$.
\end{proof}

We are now ready to prove Theorem \ref{th41}.

\begin{proof}[Proof of Theorem \ref{th41}.]

(i)$\Rightarrow$(ii): Assume $\mathcal{I}_\prec(M)$ is a complete intersection. Then $\mathcal{I}_\prec(M)=(u_1,\ldots,u_h)$, where $u_1,\ldots,u_h$ are pairwise coprime monomials. Let $C_1,\ldots,C_h$ be circuits of $M$ such that $x_{bc_\prec(C_i)}=u_i$ for $i=1,\ldots,h.$ Then the broken circuits $bc_\prec(C_i)$ are pairwise disjoint. We need to show that if $C\in \mathfrak{C}(M)$ and $bc_\prec(C)$ is a minimal broken circuit of $M$ then $C=C_j$ for some $1\leq j\leq h$. Indeed, one checks that $\{bc_\prec(C_i): i=1,\ldots,h\}$ is the set of minimal broken circuits of $M$, so $bc_\prec(C)=bc_\prec(C_j)$ for some $1\leq j\leq h$. If $C\ne C_j$, then by Theorem \ref{lm7}, there exists a circuit $C'$ of $M$ with
$$C'\subseteq C\cup C_j-\{e\}=bc_\prec(C_j)\cup \mathrm{min}_\prec(C_j)\cup \mathrm{min}_\prec(C)-\{e\},$$
where $e\in bc_\prec(C_j).$ Observe that one has either $bc_\prec(C')\subseteq C-\{e\}$ or $bc_\prec(C')\subseteq C_j-\{e\}$. From this it easily follows that $bc_\prec(C_i)\nsubseteq bc_\prec(C')$ for all $i=1,\ldots,h,$ which is a contradiction.

(ii)$\Rightarrow$(iii): Let $\mathfrak{D}$ be the subset of $\mathfrak{C}(M)$ such that $\{bc_\prec(C): C\in\mathfrak{D}\}$ is the set of minimal broken circuits of $M$. Then $\mathfrak{D}$ is simple because the minimal broken circuits of $M$ are pairwise disjoint. Since $\mathfrak{D}$ is a generating set of $\mathfrak{C}$, it follows from Lemma \ref{lm9} that
$$\mathfrak{C}(M)\subseteq\{\mathcal{C}(\mathfrak{D}'): \mathfrak{D}'\subseteq \mathfrak{D}\text{ is a tree} \}.$$
Now let $\mathfrak{D}'\subseteq \mathfrak{D}$ be a tree. By Lemma \ref{lm8}(i), there exists a circuit $C\in\mathfrak{C}(M)$ such that $C\subseteq \mathcal{C}(\mathfrak{D}').$ As we have just seen, $C=\mathcal{C}(\mathfrak{D}'')$ for some tree $\mathfrak{D}''\subseteq\mathfrak{D}.$ It then follows from Lemma \ref{lm8}(iii) that
$$\mathcal{C}(\mathfrak{D}')=\mathcal{C}(\mathfrak{D}'')=C\in\mathfrak{C}(M).$$
Therefore,
$$\mathfrak{C}(M)=\{\mathcal{C}(\mathfrak{D}'): \mathfrak{D}'\subseteq \mathfrak{D}\text{ is a tree}\}.$$

(iii)$\Rightarrow$(i): Since $\mathfrak{D}$ is simple, the monomials $x_{bc_\prec(C)}$ for $C\in\mathfrak{D}$ are pairwise coprime. Thus it suffices to show that $\mathcal{I}_\prec(M)=(x_{bc_\prec(C)}: C\in\mathfrak{D})$, or in other words, $\mathfrak{D}$ is a generating set of $\mathfrak{C}(M).$ The latter fact is, however, merely a consequence of Lemma \ref{lm10}.
\end{proof}

\begin{example}
Let $G$ be the graph in Figure \ref{fig1}, with the given numbering of the edges. Let $M=M(G)$ be the cycle matroid of $G$. We have
$$\begin{aligned}
\mathfrak{C}(M)=\big\{&\{1,2,8\},\{3,4,9\},\{5,6,10\},\{7,8,9,10\},\{1,2,9,10,7\},\{3,4,10,7,8\},\\
&\{5,6,7,8,9\},\{1,2,3,4,10,7\} ,\{3,4,5,6,7,8\},\{1,2,9,5,6,7\},\{1,2,3,4,5,6,7\}\big\}.
\end{aligned}$$
With the ordinary ordering of $\{1,\ldots,10\}$, the minimal broken circuits of $M$ are not pairwise disjoint (e.g., $\{2,8\}$ and $\{8,9,10\}$). However, this can be the case for other orderings. Consider, say, the ordering $10\prec9\prec\cdots\prec1. $ In this case, the minimal broken circuits of $M$ are $\{1,2\},\{3,4\},\{5,6\},\{7,8,9\}$, and the ideal
$$
\mathcal{I}_\prec(M)=(x_1x_2,x_3x_4,x_5x_6,x_7x_8x_9)$$
is a complete intersection.
\end{example}

\begin{figure}[htb]
\label{fig1}
 \begin{tikzpicture}[scale=1.5]
    \draw (0,0) -- (0,1) -- (2,1) -- (2,0) -- (0,0);
    \draw (0,1) -- (1,2) -- (2,1);
    \draw (2,1) -- (3,0.5) -- (2,0);
    \draw (0,0) -- (1.75,-1) -- (2,0);
    \draw (0.5,1.75) node {1};
    \draw (1.5,1.75) node {2};
    \draw (2.5,1) node {3};
    \draw (2.5,0) node {4};
    \draw (2,-0.5) node {5};
    \draw (0.5,-0.5) node {6};
    \draw (0,0.5)  node [left] {7};
    \draw (1,1)  node [below] {8};
    \draw (2,0.5)  node [left] {9};
    \draw (1,1)  node [below] {8};
    \draw (1,0)  node [above] {10};
   \end{tikzpicture}
\centerline{Figure \ref{fig1}}
\end{figure}

The above example illustrates a somewhat more general fact which holds for arbitrary triangulations of simple polygons. Recall that a simple polygon can always be partitioned into triangles by its diagonals; see, e.g., \cite[Theorem 1.2.3]{OR}.

\begin{corollary}\label{co48}
Let $G$ be a triangulation of a simple polygon. Denote by $M(G)$ the cycle matroid of $G$. Then there exists an ordering $\prec$ of the edges of $G$ such that the minimal broken circuits of the ordered matroid $(M(G),\prec)$ are pairwise disjoint.
\end{corollary}

\begin{proof}
Denote by $\mathfrak{C}$ the set of circuits of $M(G)$. Let $\mathfrak{D}$ be the subset of $\mathfrak{C}$ consisting of circuits which are boundaries of triangles of $G$. Then the intersection graph $\mathcal{G}(\mathfrak{D})$ of $\mathfrak{D}$ is a tree; see \cite[Lemma 1.2.6]{OR}. So by Remark \ref{rm43}, $\mathfrak{D}$ is a simple subset of $\mathfrak{C}$ with respect to a suitable ordering $\prec$ of the edges of $G$. It is then clear that $\mathfrak{C}$ can be described as in Theorem \ref{th41}, i.e.,
$$
\mathfrak{C}=\{\mathcal{C}(\mathfrak{D}'): \mathfrak{D}'\subseteq \mathfrak{D}\text{ is a tree}\}.
$$
Thus the minimal broken circuits of $(M(G),\prec)$, which are the broken circuits of the elements of $\mathfrak{D}$, are pairwise disjoint.
\end{proof}

In the following theorem, we verify Conjecture \ref{cj} for ordered matroids with disjoint minimal broken circuits. A formula for the Poincar\'{e} polynomials of the Orlik-Solomon algebras of those matroids is also derived. It can be considered as a generalization of a formula for the chromatic polynomials of the graphs that admit a broken circuit complex with disjoint minimal broken circuits obtained by Wilf in \cite{W}.

\begin{theorem}\label{th48}
Let $(M,\prec)$ be an ordered simple matroid on $[n]$. Assume that the minimal broken circuits of $M$ are pairwise disjoint. Then we have the following formula for the Poincar\'{e} polynomial of the Orlik-Solomon algebra of $M$:
$$\pi(\mathbf{A}(M),t)=(t+1)^{n-\sum_{i=1}^hq_i}\prod_{i=1}^h\big((t+1)^{q_i}-t^{q_i}\big),$$
where $q_1,\ldots,q_h$ are the sizes of the minimal broken circuits. Moreover, the following conditions are equivalent:
\begin{enumerate}
\item $\pi(\mathbf{A}(M),t)$ factors completely over $\mathbb{Z}$;
\item $q_i=2$ for all $i=1,\ldots,h$;
\item $M$ is supersolvable;
\item $\mathbf{A}(M)$ is Koszul.
\end{enumerate}
\end{theorem}
Note that the formula for the Poincar\'{e} polynomial can also be deduced from \cite{BrOx} since the underlying simplicial complex can be seen as an iterated join of boundaries of simplices. Then the corresponding characteristic polynomial factors nicely and one concludes by applying, e.g., \cite[Corollary 7.10.3]{B}.

\begin{proof}
Let $C_1,\ldots,C_h$ be the circuits of $M$ such that $\{bc_\prec(C_i):i=1,\ldots,h\}$ is the set of minimal broken circuits of $M$. Then $\mathcal{I}_\prec(M)=(x_{bc_\prec(C_i)}:i=1,\ldots,h)$ is a complete intersection. In this case, the Hilbert series of the ring $S/\mathcal{I}_\prec(M)$ is easily computable:
$$H_{S/\mathcal{I}_\prec(M)}(t)=\frac{\prod_{i=1}^h(1-t^{q_i})}{(1-t)^n}.$$
Proposition \ref{th26} now yields
$$\pi(\mathbf{A}(M),t)=H_{S/\mathcal{I}_\prec(M)}\Big(\frac{t}{t+1}\Big)=(t+1)^{n-\sum_{i=1}^hq_i}\prod_{i=1}^h\big((t+1)^{q_i}-t^{q_i}\big).$$
Since $(t+1)^{q_i}-t^{q_i}=\prod_{j=1}^{q_i}(t+1-\zeta^jt)$, where $\zeta$ is a primitive $q_i$th root of unity, it follows from the above equation that $\pi(\mathbf{A}(M),t)$ factors completely over $\mathbb{Z}$ if and only if $q_i=2$ for all $i=1,\ldots,h$ (note that $q_i\geq2$ as $M$ is simple). This proves (i)$\Leftrightarrow$(ii). The implication (ii)$\Rightarrow$(iii) is true for all simple matroids; see \cite[Theorem 2.8]{BZ}. It is well-known that if $M$ is supersolvable, then the Orlik-Solomon ideal $J(M)$ has a quadratic Gr\"{o}bner basis; see \cite[Theorem 7.10.2]{B} and \cite[Theorem 2.8]{BZ}. So the implication (iii)$\Rightarrow$(iv), which now follows from a result due to Fr\"{o}berg (see \cite{F}), is also true in general. Finally, in order to prove the implication (iv)$\Rightarrow$(ii), recall that the Koszulness of $\mathbf{A}(M)$ implies the quadraticity of the Orlik-Solomon ideal $J(M)$, it suffices to show that $\partial e_{C_i}$ are minimal generators of $J(M)$ for $i=1,\ldots,h$. From the description of $\mathfrak{C}(M)$ in Theorem \ref{th41} we easily get $J(M)=(\partial e_{C_i}:i=1,\ldots,h)$ (this can also be seen from a result of Bj\"{o}rner \cite[Theorem 7.10.2]{B} that $\{\partial e_{C_i}:i=1,\ldots,h\}$ is a Gr\"{o}bner basis of $J(M)$ with respect to the lexicographical order). Thus if $\partial e_{C_i}$ is not a minimal generator of $J(M)$, then
$$\partial e_{C_i}=\sum_{j\ne i}a_j \partial e_{C_j},\quad a_j\in E.$$
It follows that there must be some $j\ne i$ and some $e\in C_j$ such that $C_j-\{e\}\subseteq bc_\prec(C_i)$. But this is impossible because $bc_\prec(C_i)\cap bc_\prec(C_j)=\emptyset.$ The theorem has been proved.
\end{proof}

\begin{example}
Let $G$ be a triangulation of a simple polygon of $\ell$ vertices. Then $G$ has $2\ell-3$ edges and consists of $\ell-2$ triangles; see, e.g., \cite[Theorem 1.2.3 and Lemma 1.2.4]{OR}. By Corollary \ref{co48}, there exists an ordering $\prec$ of the edges of $G$ such that the minimal broken circuits of the matroid $(M(G),\prec)$ are pairwise disjoint. Note that all these minimal broken circuits have cardinality 2 since they come from triangles of $G$. So by Corollary \ref{th27} and Theorem \ref{th48} we obtain a known formula for the chromatic polynomial of $G$:
$$
\begin{aligned}
\chi(G,t)&=t^\ell\pi(\mathbf{A}(M(G)),-t^{-1})\\
&=t^\ell(-t^{-1}+1)^{2\ell-3-2(\ell-2)}(2(-t)^{-1}+1)^{\ell-2}=t(t-1)(t-2)^{\ell-2}.
\end{aligned}
$$
\end{example}

We now improve Wilf's upper bounds on the coefficients of the chromatic polynomial of a maximal planar graph in \cite[Theorem 4]{W}. Recall that a planar graph $G$ is called {\it maximal} if one cannot add a new edge (on the given vertex set of $G$) to form another planar graph. It is well-known that a maximal planar graph $G$ with $\ell\geq 3$ vertices has $2\ell-4$ faces, and every face of $G$ (including the outer face) is bounded by a triangle. As a key step in proving \cite[Theorem 4]{W}, Wilf showed that for a maximal planar graph $G$ with $\ell\geq 3$ vertices, there exists an ordering of the edges of $G$ so that the cycle matroid $M(G)$ has at least $\ell-2$ pairwise disjoint broken circuits, cf. \cite[Theorem 3]{W}. This can be sharpen as follows.

\begin{proposition}\label{pr49}
Let $G$ be a maximal planar graph with $\ell\geq 3$ vertices. Then there exists an ordering of the edges of $G$ so that the number of  pairwise disjoint broken circuits of $M(G)$ with respect to this ordering is bounded below by $\ell-2+\lfloor \ell/4\rfloor$. Moreover, if the dual graph of $G$ contains no triangles, then the lower bound can be improved to $\ell-3+\lceil \ell/3\rceil$.
\end{proposition}

\begin{proof}
Denote by $\mathfrak{C}$ the set of circuits of $M(G)$. Let $\mathfrak{D}$ be the subset of $\mathfrak{C}$ consisting of circuits which are boundaries of faces of $G.$ Observe that the intersection graph $\mathcal{G}(\mathfrak{D})$ is isomorphic to the dual graph of $G$. So $\mathcal{G}(\mathfrak{D})$ is a cubic graph of $2\ell-4$ vertices. We will assume that $\ell>4$ as the cases $\ell=3$ and $\ell=4$ can be easily checked (of course, one may also apply \cite[Theorem 3]{W} to these cases). Then by \cite[Theorem 4 and Theorem 5]{BHS}, there is a forest $\mathfrak{D}'\subset\mathfrak{D}$ with the cardinality at least
\begin{enumerate}[(a)]
\item
$\big(5(2\ell-4)-2\big)/8=\ell-2+(\ell-3)/4$ in the general case, and
\item
$\big(2(2\ell-4)-1\big)/3=\ell-3+\ell/3$ in the case $\mathcal{G}(\mathfrak{D})$ has no triangles.
\end{enumerate}
Now by Remark \ref{rm43}, $\mathfrak{D}'$ is a simple subset of $\mathfrak{C}$ with respect to a suitable ordering  of the edges of $G$. Since the broken circuits of the circuits in $\mathfrak{D}'$ are then pairwise disjoint, the proposition follows.
\end{proof}

The above proposition yields the following improvement of \cite[Theorem 4]{W}.

\begin{theorem}\label{th411}
Let $\chi(G,t)=t^\ell+\sum_{p=1}^{\ell-1}(-1)^pa_{p}t^{\ell-p}$ be the chromatic polynomial of a maximal planar graph $G$. Then the coefficients of $\chi(G,t)$ are dominated by the corresponding coefficients of the function
$$Q(t)=t^{-\ell+4+\lfloor \ell/4\rfloor}(t+1)^{\ell-2-2\lfloor \ell/4\rfloor}(t+2)^{\ell-2+\lfloor \ell/4\rfloor},$$
or explicitly,
$$a_p\leq\sum_{k=0}^{\min\{p,\ell-2+\lfloor \ell/4\rfloor\}}\binom{\ell-2-2\lfloor \ell/4\rfloor}{p-k}\binom{\ell-2+\lfloor \ell/4\rfloor}{k}2^k,\quad p=1,\ldots,\ell-1.$$
In the case when the dual graph of $G$ has no triangles, the function $Q(t)$ can be replaced by
$$R(t)=t^{-\ell+3+\lceil \ell/3\rceil}(t+1)^{\ell-2\lceil \ell/3\rceil}(t+2)^{\ell-3+\lceil \ell/3\rceil},$$
and we have
$$a_p\leq\sum_{k=0}^{\min\{p,\ell-3+\lceil \ell/3\rceil\}}\binom{\ell-2\lceil \ell/3\rceil}{p-k}\binom{\ell-3+\lceil \ell/3\rceil}{k}2^k,\quad p=1,\ldots,\ell-1.$$
\end{theorem}

\begin{proof}
One only needs to replace \cite[Theorem 3]{W} by Proposition \ref{pr49} in the proof of \cite[Theorem 4]{W}.
\end{proof}

We end this subsection with an examination of 3-codimensional Stanley-Reisner ideals of broken circuit complexes. We show that for those ideals, Gorensteiness is equivalent to be a complete intersection. Gorenstein ideals of codimension 3 were classified in Buchsbaum-Eisenbud's structure theorem \cite[Theorem 2.1]{BE}. Bruns-Herzog \cite[Theorem 6.1]{BH} and Kamoi \cite[Theorem 0.1]{K} then independently refined this classification for monomial ideals. They showed that if $I\subset S={K}[x_1,\ldots,x_n]$ is a monomial Gorenstein ideal of codimension 3 with $m$ minimal generators ($m$ is odd by \cite[Theorem 2.1]{BE}), then there are $m$ pairwise coprime monomials $u_1,\ldots,u_m$ of $S$ such that $I$ is generated by the monomials
$$v_i=u_iu_{i+1}\cdots u_{i+s-1},\quad i=1,\ldots,m,$$
where $s=(m-1)/2$ and $u_j=u_{j-m}$ for $j>m.$

\begin{proposition}\label{pr412}
Let $(M,\prec)$ be an ordered simple matroid on $[n]$ and let $\mathcal{I}_\prec(M)\subset S$ be the Stanley-Reisner ideal of the broken circuit complex of $M$. Assume that $\codim \mathcal{I}_\prec(M)\leq 3$. Then $\mathcal{I}_\prec(M)$ is Gorenstein if and only if it is a complete intersection.
\end{proposition}

\begin{proof}
The proposition is true for all ideals of codimension 1 and codimension 2; see \cite[Corollary 21.10]{E}. Therefore, it suffices to prove that $\mathcal{I}_\prec(M)$ is a complete intersection when it is a Gorenstein ideal of codimension 3. Let $m$ be the number of minimal generators of $\mathcal{I}_\prec(M)$. Set $s=(m-1)/2$ and let $u_1,\ldots,u_m$ be pairwise coprime monomials such that $I=(v_1,\ldots,v_m)$, where
$v_i=u_iu_{i+1}\cdots u_{i+s-1}$ for $i=1,\ldots,m$ ($u_j=u_{j-m}$ if $j>m$).
We need to show that $m=3$. Suppose on the contrary that $m>3$. Put $\min_\prec(u_i)=\min_\prec\{j:x_j\mid u_i\}$ for $i=1,\ldots,m.$ We may assume
$$t:=\mathrm{min}_\prec(u_1)=\mathrm{min}_\prec\{\mathrm{min}_\prec(u_i): i=1,\ldots,m\}.$$
Let $C_i$ be the circuits of $M$ such that $v_i=x_{bc_\prec(C_i)}$ for $i=1,\ldots,m.$ Note that $u_1\mid v_m$ since $s>1$. So the above assumption yields
$$t=\mathrm{min}_\prec\big(bc_\prec(C_1)\big)=\mathrm{min}_\prec\big(bc_\prec(C_m)\big)=\mathrm{min}_\prec\{\mathrm{min}_\prec\big(bc_\prec(C_i)\big): i=1,\ldots,m\}.$$
Hence we can find $p,q\prec t$ such that $C_1=bc_\prec(C_1)\cup \{p\}$ and $C_m=bc_\prec(C_m)\cup \{q\}$. By Theorem \ref{lm7}, there exists a circuit $C$ of $M$ with $C\subseteq C_1\cup C_m-\{t\}$. Since $x_{bc_\prec(C)}\in \mathcal{I}_\prec(M)$, it follows that $x_{C_1\cup C_m-\{t\}}\in \mathcal{I}_\prec(M)$. We have
$$x_{C_1\cup C_m-\{t\}}=\frac{u_1}{x_t}u_2\cdots u_{s-1}u_su_mx_px_q.$$
As $m>s+1$ and $(v_i,x_px_q)=1$ for $i=1,\ldots,m$, it is easy to check that $v_i\nmid x_{C_1\cup C_m-\{t\}}$  for all $i=1,\ldots,m$. This implies $x_{C_1\cup C_m-\{t\}}\not\in \mathcal{I}_\prec(M)$, a contradiction.
\end{proof}

\subsection{Orlik-Terao ideals}

In the following we will characterize arrangements whose Orlik-Terao ideal is a complete intersection and show that Conjecture \ref{cj} holds for those arrangements. We begin with a simple lemma. It is known, but due to the lack of reference we present a proof here.

\begin{lemma}\label{lm413}
Let $I\subset S={K}[x_1,\ldots,x_n]$ be a graded prime ideal which is minimally generated by homogeneous polynomials $u_1,\ldots,u_h$. Then the following conditions are equivalent:
\begin{enumerate}
\item $I$ is a complete intersection;
\item Every subset of $u_1,\ldots,u_h$ generates a prime ideal.
\end{enumerate}
\end{lemma}

\begin{proof}
(i)$\Rightarrow$(ii): Let $R$ be a quotient ring of $S$ by a graded ideal. By descending induction it is enough to show that if there is a homogeneous regular element $u\in R$ such that $(u)$ is a prime ideal, then $R$ is a domain. Indeed, let $P\subset (u)$ be a minimal prime ideal of $R$. Then for each $v\in P$ we have $v=uw$ with $w\in R$. Since $u\not\in P$, $w$ is an element of $P$. It follows that $P=uP$, and hence $P=0$ by Nakayama's lemma. Therefore, $R$ is a domain.

(ii)$\Rightarrow$(i): Let $I_j=(u_1,\ldots, u_j)$ for $j=1,\ldots,h$. Then we have a chain of prime ideals:
$$0\subset I_1\subset\cdots\subset I_h=I.$$
This chain is strict because of the minimality of the set of generators $\{u_1,\ldots,u_h\}$. Hence $\codim I=h$, from which follows that $I$ is a complete intersection.
\end{proof}

As before, let $\mathcal{A}$ be an essential central arrangement of $n$ hyperplanes in a vector space $V$ over a field ${K}$. Let $M(\mathcal{A})$ be the underlying matroid and $I(\mathcal{A})\subset S={K}[x_1,\ldots,x_n]$ the Orlik-Terao ideal of $\mathcal{A}$. Denote by $\mathfrak{C}=\mathfrak{C}\big(M(\mathcal{A})\big)$ the set of circuits of $M(\mathcal{A})$. Recall that each circuit $C$ of $M(\mathcal{A})$ corresponds to a unique (up to a scalar multiple) relation $r_C$ in the relation space $F(\mathcal{A})$. The ideal $I(\mathcal{A})$ is then generated by the polynomials $\{\partial r_C:C\in\mathfrak{C}\}$; see Theorem \ref{th24}. In the following we will sometimes make use of the fact that $I(\mathcal{A})$ is a prime ideal (see \cite[Proposition 2.1]{ST}) without mentioning it explicitly.

\begin{lemma}\label{lm414}
Assume that $I(\mathcal{A})$ is a complete intersection. Let $\mathfrak{D}$ be a subset of $\mathfrak{C}$ such that $\{\partial r_{C},C\in\mathfrak{D}\}$ is a minimal system of generators of $I(\mathcal{A})$. Then for every subset $\mathfrak{D}'$ of $\mathfrak{D}$, if a relation $r$ belongs to the subspace of $F(\mathcal{A})$ generated by $\{r_C, C\in\mathfrak{D}'\}$, then $\partial r$ is an element of the ideal of $R$ generated by $\{\partial r_C, C\in\mathfrak{D}'\}.$
\end{lemma}

\begin{proof}
For each relation $r=\sum_{i=1}^na_ix_i\in F(\mathcal{A})$, set $\Lambda_r=\{i\in [n]:a_i\ne 0\}$. Assume now that $r=\sum_{C\in\mathfrak{D}'}a_Cr_C,\ a_C\in{K}$. Substituting $(1/x_1,\ldots,1/x_n)$ in this equation we get
$$\frac{\partial r}{x_{\Lambda_r}}=\sum_{C\in\mathfrak{D}'}\frac{a_C\partial r_C}{x_{C}}.$$
It follows that
$$x_{[n]-\Lambda_r}\partial r=\sum_{C\in\mathfrak{D}'}a_Cx_{[n]-C}\partial r_C\in(\partial r_C: C\in\mathfrak{D}').$$
By Lemma \ref{lm413}, $I_{\mathfrak{D}'}=(\partial r_C: C\in\mathfrak{D}')$ is a prime ideal. Since $I_{\mathfrak{D}'}$ is generated in degree $\geq2$, $x_{[n]-\Lambda_r}\not \in I_{\mathfrak{D}'}$. Hence, $\partial r\in  I_{\mathfrak{D}'}.$
\end{proof}

We are now in a position to prove the following characterizations of the complete intersection property of the Orlik-Terao ideal.

\begin{theorem}\label{th415}
Let $\mathcal{A}$ be an essential central arrangement of $n$ hyperplanes. The following conditions are equivalent:
\begin{enumerate}
\item
$I(\mathcal{A})$ is a complete intersection;
\item
There is an ordering $\prec$ of $[n]$, with an arbitrary induced monomial order on $R$, such that $\mathrm{in}_\prec(I(\mathcal{A}))$ is a complete intersection;
\item
There is an ordering $\prec$ of $[n]$ such that the minimal broken circuits of $(M(\mathcal{A}),\prec)$ are pairwise disjoint;
\item
There is an ordering $\prec$ of $[n]$ and a subset $\mathfrak{D}$ of $\mathfrak{C}$ which is simple with respect to $\prec$ such that $\mathfrak{C}=\{\mathcal{C}(\mathfrak{D}'): \mathfrak{D}'\subseteq \mathfrak{D}\text{ is a tree}\}.$
\end{enumerate}
\end{theorem}

\begin{proof}
By Theorem \ref{th24} and Theorem \ref{th41}, we only need to prove the implication (i)$\Rightarrow$(ii). Assume that $I(\mathcal{A})$ is a complete intersection. Let $\mathfrak{D}$ be a subset of $\mathfrak{C}$ such that $\{\partial r_{C},C\in\mathfrak{D}\}$ is a minimal system of generators of $I(\mathcal{A})$. We will show that there is an ordering of $[n]$ so that $\mathfrak{D}$ is simple with respect to this ordering. By Remark \ref{rm43}, this will be done after the following two claims have been proved.
\begin{claim}\label{cl1}
 $|C\cap C'|\leq1$ for all $C,C'\in \mathfrak{D}.$
\end{claim}
If this is not the case, then there are distinct elements $p,q\in C_1\cap C_2$ for some $C_1,C_2\in \mathfrak{D}.$ We may assume $r_{C_i}=x_p+\sum_{j\in C_i-\{p\}}a_{ji}x_j$ for $i=1,2$. Then the relation $r=r_{C_1}-r_{C_2}$ does not involve $x_p$. By Lemma \ref{lm414}, $\partial r=f_1\partial r_{C_1}+f_2\partial r_{C_2}$ for some polynomials $f_1,f_2\in S$. Write $f_i=g_i+x_ph_i$ with $g_i,h_i\in S$ and $g_i$ does not involve $x_p$. We have
$$
\begin{aligned}
\partial r&=f_1\partial r_{C_1}+f_2\partial r_{C_2}\\
&=(g_1+x_ph_1)(x_{C_1-\{p\}}+x_pk_1)+(g_2+x_ph_2)(x_{C_2-\{p\}}+x_pk_2)\quad (k_1,k_2\in S)\\
&=g_1x_{C_1-\{p\}}+g_2x_{C_2-\{p\}}+x_pl \quad (l\in S).
\end{aligned}
$$
This yields $\partial r=g_1x_{C_1-\{p\}}+g_2x_{C_2-\{p\}}$ since $\partial r$ does not involve $x_p$. It follows that $x_q\mid \partial r$. But this is impossible by the definition of $\partial r$.
\begin{claim}
The intersection graph $\mathcal{G}(\mathfrak{D})$ does not have a cycle with pairwise distinct edge labels, i.e., a cycle $C_1\ldots C_m$ with $C_i\cap C_{i+1}\ne C_j\cap C_{j+1}$ whenever $i\ne j.$
\end{claim}
We use a similar argument as in the proof of Claim \ref{cl1}. Suppose that $\mathcal{G}(\mathfrak{D})$ contains cycles with pairwise distinct edge labels. Let $C_1\ldots C_m$ be such a cycle with shortest length. Then it is easily seen that $C_i\cap C_j=\emptyset$ for $i=1,\ldots,m$ and $j\ne i-1,i+1$ ($C_{m+1}=C_1$). Let $C_i\cap C_{i+1}=\{p_i\}$ for $i=1,\ldots,m$. Recall that the relations $r_{C_i}$ are determined up to a scalar multiple. So we may choose these relations such that the relation $r=\sum_{i=1}^mr_{C_i}$ does not involve $x_{p_1},\ldots,x_{p_{m-1}}$. By Lemma \ref{lm414}, $\partial r=\sum_{i=1}^mf_i\partial r_{C_i}$ for some $f_i\in S$. Let $P$ be the ideal of $S$ generated by $x_{p_1},\ldots,x_{p_{m-1}}$. Write $f_i=g_i+h_i$ with $h_i\in P$ and $g_i$ does not involve $x_{p_1},\ldots,x_{p_{m-1}}$ for $i=1,m$. Note that $\partial r_{C_i}\in P$ for $i=2,\ldots,m-1$. We have
$$\begin{aligned}
   \partial r=&\ f_1\partial r_{C_1}+f_m\partial r_{C_m}+\sum_{i=2}^{m-1}f_i\partial r_{C_i}\\
=&\ (g_1+h_1)(x_{C_1-\{p_1\}}+x_{p_1}k)+(g_m+h_m)(x_{C_m-\{p_{m-1}\}}+x_{p_{m-1}}k')\\
&+\sum_{i=2}^{m-1}f_i\partial r_{C_i}\quad (k,k'\in S)\\
=&\ g_1x_{C_1-\{p_1\}}+g_mx_{C_m-\{p_{m-1}\}}+l \quad (l\in P).
  \end{aligned}
$$
Similarly as in the proof of Claim \ref{cl1}, this implies $\partial r=g_1x_{C_1-\{p_1\}}+g_mx_{C_m-\{p_{m-1}\}}$, and hence $x_{p_m}\mid \partial r$, which is impossible.

Now assume that $\mathfrak{D}$ is a simple subset of $\mathfrak{C}$ with respect to an ordering $\prec$ of $[n]$. We denote an induced monomial order of $\prec$ on $S$ by the same notation. Then the monomials $\mathrm{in}_\prec(\partial r_C)$ are pairwise coprime for all $C\in\mathfrak{D}$. It follows that $\{\partial r_C:C\in\mathfrak{D}\}$ is a Gr\"{o}bner basis of $I(\mathcal{A})$; see, e.g., \cite[Corollary 2.3.4]{HH}. Hence $\mathrm{in}_\prec(I(\mathcal{A}))=(\mathrm{in}_\prec(\partial r_C):C\in\mathfrak{D})$ is a complete intersection.
\end{proof}

The following corollary follows immediately from the above theorem and \cite[Proposition 1.1]{C}.
\begin{corollary}
Let $\mathcal{A}$ be an essential central arrangement of $n$ hyperplanes. Assume that $I(\mathcal{A})$ is a complete intersection. Then there exists an ordering $\prec$ of $[n]$ (with an arbitrary induced order on $S$) such that
$$
\mathrm{in}_\prec(I(\mathcal{A})^i)=\mathrm{in}_\prec(I(\mathcal{A}))^i \text{ for all } i\geq 1.
$$
\end{corollary}

Finally, we verify Conjecture \ref{cj} for arrangements with complete intersection Orlik-Terao ideal. For those arrangements several properties coincide. Recall that the arrangement $\mathcal{A}$ is said to be {\it 2-formal} if the relation space $F(\mathcal{A})$ is spanned by relations corresponding to 3-element circuits; see \cite{FR}.

\begin{corollary}\label{th416}
Let $\mathcal{A}$ be an essential central arrangement of $n$ hyperplanes. Assume that the Orlik-Terao ideal $I(\mathcal{A})$ of $\mathcal{A}$ is a complete intersection. Let $q_1,\ldots,q_h$ be the degree sequence of a minimal system of homogeneous generators of $I(\mathcal{A})$. Then the Poincar\'{e} polynomial of the Orlik-Solomon algebra of $\mathcal{A}$ is
$$
\pi(\mathbf{A}(\mathcal{A}),t)=(t+1)^{n-\sum_{i=1}^hq_i}\prod_{i=1}^h\big((t+1)^{q_i}-t^{q_i}\big).
$$
Moreover, the following conditions are equivalent:
\begin{enumerate}
\item $\pi(\mathbf{A}(\mathcal{A}),t)$ factors completely over $\mathbb{Z}$;
\item $q_i=2$ for all $i=1,\ldots,h$;
\item $\mathcal{A}$ is supersolvable;
\item $\mathcal{A}$ is free;
\item $\mathcal{A}$ is 2-formal;
\item $\mathbf{A}(\mathcal{A})$ is Koszul;
\item $\mathbf{C}(\mathcal{A})$ is Koszul.
\end{enumerate}
\end{corollary}

\begin{proof}
Let $C_1,\ldots,C_h$ be circuits of $M(\mathcal{A})$ such that $\{\partial r_{C_i}:i=1,\ldots,h\}$ is a minimal set of generators of $I(\mathcal{A})$. It follows from the proof of Theorem \ref{th415} that for a suitable ordering $\prec$ of $[n]$ the ideal $\mathrm{in}_\prec(I(\mathcal{A}))$ is a complete intersection and is minimally generated by $\{\mathrm{in}_\prec(\partial r_{C_i}):i=1,\ldots,h\}$. In particular, the minimal broken circuits of the matroid $(M(\mathcal{A}),\prec)$ are pairwise disjoint and have the sizes $q_1,\ldots,q_h$. The formula for the Poincar\'{e} polynomial and the equivalence of conditions (i), (ii), (iii), (vi) then follow from Theorem \ref{th48}. For the equivalence of (ii) and (vii), one only needs to notice that $q_1,\ldots,q_h$ is the degree sequence of a minimal system of homogeneous generators of both $I(\mathcal{A})$ and $\mathrm{in}_\prec(I(\mathcal{A}))$; see \cite{F}. It is well-known that the implications (iii)$\Rightarrow$(iv)$\Rightarrow$(i) and (iv)$\Rightarrow$(v) are true in general; see \cite[Theorem 4.58, Theorem 4.137]{OT} and \cite[Corollary 2.5]{Y3}. To complete the proof, we will show (v)$\Rightarrow$(ii). Assume that $\mathfrak{D}=\{C_1,\ldots,C_m\}$ ($m\leq h$) is the subset of $\{C_1,\ldots,C_h\}$ consisting of 3-element circuits. Let $I_\mathfrak{D}=(\partial r_{C_i}:i=1,\ldots,m)$. This ideal is prime by Lemma \ref{lm413}. One easily sees that $\partial r_C\in I_\mathfrak{D}$ for every 3-element circuit $C$ of $M(\mathcal{A})$. Since $\mathcal{A}$ is 2-formal, for any relation $r\in F(\mathcal{A})$ we have $r=\sum_{C\in\mathfrak{D}'}a_Cr_C,$ where $a_C\in{K}$ and $\mathfrak{D}'$ is a set of 3-element circuits. It follows from the proof of Lemma \ref{lm414}, with the notation used there, that $x_{[n]-\Lambda_r}\partial r\in(\partial r_C: C\in\mathfrak{D}').$ Hence $x_{[n]-\Lambda_r}\partial r\in I_\mathfrak{D}$. This implies $\partial r\in I_\mathfrak{D}$ since $I_\mathfrak{D}$ is a prime ideal generated in degree $\geq 2$.  Therefore, $I(\mathcal{A})=(\partial r:r\in F(\mathcal{A}))=I_\mathfrak{D}$, or, in other words, $\mathfrak{D}=\{C_1,\ldots,C_h\}$. So we obtain $q_i=2$ for all $i=1,\ldots,h$, as desired.
\end{proof}

\begin{remark}
Denham, Garrousian and Tohaneanu have independently proved the equivalence of conditions (ii), (iii), (v), (vii) in Corollary \ref{th416} by a different method; see \cite[Corollary 5.12]{DGT}.
\end{remark}

\end{document}